\documentclass[reqno]{amsart}

\usepackage{verbatim} 
\usepackage{amsmath}
\usepackage{amsfonts}
\usepackage{amssymb}
\usepackage{mathrsfs}
\usepackage{amsthm}
\usepackage{newlfont}
\usepackage{verbatim}

\usepackage{fullpage,amsmath, amssymb,amscd}
\usepackage{latexsym, epsfig, color}
\usepackage[mathscr]{eucal}
\usepackage{xypic, graphicx}
\usepackage{amscd}
\usepackage[all, cmtip]{xy}

\newcommand\cyr{%
 \renewcommand\rmdefault{wncyr}%
 \renewcommand\sfdefault{wncyss}%
 \renewcommand\encodingdefault{OT2}%
\normalfont\selectfont} \DeclareTextFontCommand{\textcyr}{\cyr}

\newtheorem{theorem}{Theorem}
\newtheorem{lemma}[theorem]{Lemma}

\newtheorem{proposition}[theorem]{Proposition}

\theoremstyle{remark}
 \newtheorem{remark}[theorem]{Remark}

\def\Z{\mathbb Z}
\def\N{\mathbb N}
\def\Q{\mathbb Q}

\def\C{\mathbb C}
\def\F{\mathbb F}

\def\cP{\mathcal P}
\def\cO{\mathcal O}
\def\cE{\mathcal E}
\def\cF{\mathcal F}
\def\cM{\mathcal M}
\def\fp{\mathfrak p}

\def\fP{\mathfrak P}

\def\deg{\operatorname{deg}}
\def\det{\operatorname{det}}

\def\End{\operatorname{End}}

\def\Gal{\operatorname{Gal}}

\def\Cl{\operatorname{Cl}}

\def\mod{\operatorname{mod}}

\def\disc{\operatorname{disc}}

\def\gcd{\operatorname{gcd}}

\def\GL{\operatorname{GL}}
\def\SL{\operatorname{SL}}
\def\PGL{\operatorname{PGL}}

\def\O{\operatorname{O}}
\def\o{\operatorname{o}}

\def\log{\operatorname{log}}

\def\ord{\operatorname{ord}}

\def\rank{\operatorname{rank}}

\def\ds{\displaystyle}
\def\Tr{\operatorname{Tr}}

\newcommand{\Fi}{F_\infty}
\newcommand{\Ci}{\mathbb{C}_\infty}
\newcommand{\alg}{\mathrm{alg}}

\begin{document}

\title{The growth of the discriminant of the endomorphism ring of the reduction of a rank 2 generic Drinfeld module}


\author{
Alina Carmen Cojocaru and Mihran Papikian}
\address[Alina Carmen  Cojocaru]{
\begin{itemize}
\item[-]
Department of Mathematics, Statistics and Computer Science, University of Illinois at Chicago, 851 S Morgan St, 322
SEO, Chicago, 60607, IL, USA;
\item[-]
Institute of Mathematics  ``Simion Stoilow'' of the Romanian Academy, 21 Calea Grivitei St, Bucharest, 010702,
Sector 1, Romania
\end{itemize}
} \email[Alina Carmen  Cojocaru]{cojocaru@uic.edu}
\address[Mihran Papikian]{
\begin{itemize}
\item[-]
Department of Mathematics, 
Pennsylvania State University,
University Park, PA 16802, USA
\end{itemize}
} \email[Mihran Papikian]{papikian@psu.edu}

\renewcommand{\thefootnote}{\fnsymbol{footnote}}
\renewcommand{\thefootnote}{\arabic{footnote}}

\renewcommand{\thefootnote}{\fnsymbol{footnote}}
\renewcommand{\thefootnote}{\arabic{footnote}}

\thanks{A.C.C. was partially supported  by  a Collaboration Grant for Mathematicians from the Simons Foundation  
under Award No. 318454.}

\thanks{M.P.  was partially supported by a Collaboration Grant for Mathematicians from the Simons Foundation 
under Award No. 637364.}

\begin{abstract}
Let $\psi: A \to F\{\tau\}$ be a Drinfeld $A$-module over $F$ of rank 2 and without complex multiplication, 
where $A = \F_q[T]$, $F = \F_q(T)$,
and $q$ is an odd prime power.
For a prime $\fp = p A$  of $A$ of good reduction for $\psi$ and with residue field $\F_{\fp}$, we study the growth of the 
absolute value $|\Delta_{\fp}|$ of the discriminant of the $\F_{\fp}$-endomorphism ring of the reduction of
$\psi$ modulo $\fp$. We prove that for all $\fp$,  $|\Delta_{\fp}|$ grows with $|p|$.
Moreover, we prove that for a density 1 of primes $\fp$,
$|\Delta_{\fp}|$ is as close as possible to its upper bound $|a_{\fp}^2 - 4 \mu_{\fp}p|$,
where
$X^2+a_{\fp}X+\mu_{\fp}  p \in A[X]$
is the characteristic polynomial of  $\tau^{\deg p}$.
\end{abstract}

\maketitle


\section{Introduction}

Let $\F_q$ be a finite field with $q$ elements, 
$A:=\F_q[T]$ be the ring of polynomials in $T$ over $\F_q$, 
$F:=\F_q(T)$ be the field of fractions of $A$,
and $F^\alg$ be a fixed algebraic closure of $F$.  
We call a nonzero prime ideal of $A$ simply a \textit{prime} of $A$. 
The main results of this paper concern the 
reductions modulo the primes of $A$ of a fixed Drinfeld module over $F$. To state these results, we 
first recall some basic concepts from the theory of Drinfeld modules. 

An \textit{$A$-field} is a field $L$ equipped with a homomorphism $\gamma: A\to L$. 
Two $A$-fields of particular prominence in this paper are $F$ and $\F_\fp:=A/\fp$, 
where $\fp\lhd A$ is a prime. 
When $L$ is an extension of $F$, 
we always  implicitly assume that $\gamma: A\hookrightarrow L$ is obtained from the natural embedding of $A$ into its field of fractions 
$A\hookrightarrow F \hookrightarrow L$;
when $L$ is a finite extension of $\F_\fp$, 
we always  implicitly assume that $\gamma: A\to A/\fp\hookrightarrow L$ is obtained 
from the natural quotient map. 

For an $A$-field $L$,  denote by $L\{\tau\}$ 
 the noncommutative ring of polynomials in $\tau$ with coefficients in $L$ 
and subject to the commutation rule $\tau c=c^q\tau$, $c\in L$. 
A \textit{Drinfeld module of rank $r\geq 1$ defined over $L$} is a ring homomorphism 
$\psi: A \to L\{\tau\} $, $a\mapsto\psi_a$, 
uniquely determined by the image of $T$:
$$
\psi_T=\gamma(T)+\sum_{1 \leq i \leq r} g_i(T) \tau^i,\quad g_r(T)\neq 0. 
$$
The \textit{endomorphism ring} of $\psi$ is the centralizer of the image $\psi(A)$ of $A$ in $L\{\tau\}$:
\begin{align*}
\End_L(\psi) &:=\{u\in L\{\tau\}: u\psi_a=\psi_a u \text{ for all }a\in A\}\\  & =\{u\in L\{\tau\}: u\psi_T=\psi_T u\}. 
\end{align*}
As $\End_L(\psi)$ contains $\psi(A)\cong A$ in its center, it is an $A$-algebra. 
It can be shown that $\End_L(\psi)$ is a free $A$-module of rank $\leq r^2$; see \cite{Dr74}. 

Now let $\psi: A\to F\{\tau\}$ be a Drinfeld module of rank $r$ over $F$ defined by 
$$
\psi_T=T+g_1\tau+\ldots+g_r\tau^r. 
$$
We say that a prime $\fp\lhd A$ is a \textit{prime of good reduction} for $\psi$ 
if $\ord_\fp(g_i)\geq 0$ for  all $1\leq i\leq r-1$
and 
if $\ord_\fp(g_r)=0$. 
In that case, we can consider $g_1, \dots, g_r$ as elements of the completion $A_\fp$ of $A$ at $\fp$. 
The \textit{reduction of $\psi$ at $\fp$} is the 
Drinfeld module $\psi\otimes\F_\fp: A \to \F_\fp\{\tau\}$ given by 
$$
(\psi\otimes\F_\fp)_T=\overline{T}+\overline{g_1}\tau+\cdots+\overline{g_r}\tau^r,  
$$
where $\overline{g}$ is the image of $g\in A_\fp$ under the canonical homomorphism $A_\fp\to A_\fp/\fp$. 
Note that $\psi\otimes\F_\fp$ has rank $r$ since $\overline{g_r}\neq 0$.
It can be shown that
 all but finitely many primes 
of $A$ are primes of good reduction for $\psi$; we denote the set of these primes by $\cP(\psi)$. 

Let $\fp=pA$ be a prime of good reduction of $\psi$, where $p$ denotes the monic generator of $\fp$. 
Denote by $\cE_{\psi, \fp}=\End_{\F_\fp}(\psi\otimes\F_\fp)$
the endomorphism ring of $\psi \otimes \F_\fp$. 
It is easy to see that $\pi_\fp:=\tau^{\deg p}$ is in the center of $\F_\fp\{\tau\}$, hence $\pi_\fp\in \cE_{\psi, \fp}$. 
Using the theory of Drinfeld modules over finite fields, it is easy to show that $A[\pi_\fp]$ and $\cE_{\psi, \fp}$ 
are $A$-orders in the imaginary field extension $F(\pi_\fp)$ of $F$ of degree $r$ (``imaginary" means that 
there is a unique place of $F(\pi_\fp)$ over the place $\infty:=1/T$ of $F$); see \cite{GaPa18}. 
Here we remind the reader that $A[\pi_\fp]=\psi(A)[\pi_\fp]$, i.e., $A\cong \psi(A)$ also denotes the image of $A$ under $\psi$. 
Then, denoting by $\cO_{F(\pi_\fp)}$ the integral closure of $A$ in $F(\pi_\fp)$, we obtain a natural inclusion of $A$-orders 
$$A[\pi_\fp]\subseteq \cE_{\psi, \fp} \subseteq \cO_{F(\pi_\fp)}.$$ 

It is an interesting problem, with important applications to the arithmetic of $F$, to compare the above orders 
as $\fp$ varies. For example, it is proved in \cite{CoPa15} (for $r=2$) and in \cite{GaPa18}, \cite{GaPa19} (for $r\geq 2$) 
that  $\cE_{\psi, \fp}/A[\pi_\fp]$ captures the splitting behavior of $\fp$ in the division fields of $\psi$, thus leading 
to (non-abelian) reciprocity laws. Also, in \cite{GaPa18}, it is shown that $\cE_{\psi, \fp}/A[\pi_\fp]$ and
 $\cO_{F(\pi_\fp)}/\cE_{\psi, \fp}$ 
can be arbitrarily large as $\fp$ varies, whereas in \cite{CoPa15} an explicit formula is given for the density of primes for which 
$A[\pi_\fp]=\cE_{\psi, \fp}$. 

In this paper, we are interested in the growth of the discriminant of $\cE_{\psi, \fp}$ as $\fp$ varies, 
and in applications of this growth to the arithmetic of $\psi$. 
Our results assume that $q$ is odd and $r=2$. These assumptions are to be kept from here on without further notice.

Choosing a basis $\cE_{\psi, \fp}=A\alpha_1+A\alpha_2$ of $\cE_{\psi, \fp}$ as a free $A$-module of rank $2$, 
the discriminant 
$\Delta_\fp:=\Delta_{\psi, \fp}$ of $\cE_{\psi, \fp}$ is $\det(\Tr_{F(\pi_\fp)/F}(\alpha_i\alpha_j))_{1\leq i, j\leq 2}$, which is well-defined up to a multiple 
by a square in $\F_q^\times$. 
If we write 
$\Delta_{\fp}=c_{\fp}^2\cdot \Delta_{F(\pi_\fp)}$, 
where $c_{\fp}, \Delta_{F(\pi_\fp)} \in A$  with $\Delta_{F(\pi_\fp)}$ square-free,
 then
$\cE_{\psi, \fp}=A+c_{\fp}\cO_{F(\pi_\fp)}$, $\Delta_{F(\pi_\fp)}$ is the discriminant of $\cO_{F(\pi_\fp)}$, and 
$\cO_{F(\pi_\fp)}/\cE_{\psi, \fp}\cong A/c_{\fp} A$ as $A$-modules.
Note that, similarly to $\Delta_{\fp}$, the polynomial $c_\fp$ also depends on $\psi$, although this will  
not be explicitly indicated in our notation.

Denote by $|\cdot|=|\cdot|_\infty$ the absolute value on $F$ corresponding to $1/T$, normalized 
so that that  $|a|=q^{\deg a}$ for $a\in A$, with $\deg a$ denoting the degree of $a$ as a polynomial in $T$
and
subject to the convention that $\deg 0=-\infty$. 
The first main result of this paper is the following: 

\begin{theorem}\label{thm1} Assume $\End_{F^\alg}(\psi)=A$. Then
	$$
	|\Delta_\fp|\gg_{\psi} \frac{\log |p|}{(\log\log |p|)^2},
	$$
	where the implied $\gg_{\psi}$-constant depends on $q$ and on the coefficients of the polynomial 
	$\psi_T \in F\{\tau\}$.
\end{theorem}

By considering $\tau$ as the Frobenius 
automorphism of $\F_\fp$ relative to $\F_q$, that is, the map $\alpha\mapsto \alpha^q$, 
we can also consider $\F_\fp$ as an $A$-module via $\psi\otimes \F_\fp$ (hence $T$ acts on $\F_\fp$ as $(\psi\otimes \F_\fp)_T$). 
This module will be denoted  ${^\psi}\F_\fp$. 
The properties of the torsion elements of this module
lead to an $A$-module isomorphism
$$
{^\psi}\F_\fp\cong A/d_{1, \fp}A\times A/d_{2, \fp}A
$$
for uniquely determined nonzero monic polynomials $d_{1, \fp}, d_{2, \fp}\in A$ such that $d_{1, \fp}\mid d_{2, \fp}$; 
as above, the polynomials  $d_{1, \fp}, d_{2, \fp}$ also depend on $\psi$, although this will not be explicitly indicated in our notation.
Note that
 $d_{2, \fp}$ may be regarded as the exponent of the $A$-module ${^\psi}\F_\fp$,
 and that $|d_{1, \fp}\cdot d_{2,\fp}|=|p|$;
thus  we have the trivial lower bound $|d_{2,\fp}|\geq |p|^{\frac{1}{2}}$. Theorem \ref{thm1} allows us to deduce a stronger lower 
bound on $|d_{2,\fp}|$:

\begin{theorem}\label{thm2}
	Assume $\End_{F^\alg}(\psi)=A$. Then 
	$$
	|d_{2, \fp}|\gg_{\psi}   |p|^{\frac{1}{2}} \frac{ (\log|p|)^{\frac{1}{2}}  }{\log\log |p|},
	$$
	where the implied $\gg_{\psi}$-constant depends on $q$ and on the coefficients of the polynomial 
	$\psi_T \in F\{\tau\}$.
\end{theorem}

Theorems \ref{thm1} and \ref{thm2} are Drinfeld module
analogues of Schoof's results for 
elliptic curves over $\Q$; see \cite{Sc91}. 
While our proof of Theorem \ref{thm1} is inspired by Schoof's argument,
it relies not only on Drinfeld's now classical function field analogue of the analytic theory of elliptic curves,
but also on the function field counterpart of the $j$-function, investigated by Gekeler  \cite{Ge97}, 
and on the Drinfeld module analogue of Deuring's lifting lemma, 
proved in an earlier paper by the present authors \cite{CoPa15}.

\begin{remark}\label{remark-complementary}
According to Theorem \ref{thm1}, $|\Delta_{\fp}|$ grows with $\deg p$. 
In relation to the growth of $|\Delta_{\fp}|$,
Theorem 1.1 of \cite{GaPa18} implies that, 
for any fixed number $\kappa > 0$, we can find $\fp$ such that $|c_{\fp}| > \kappa$;
therefore, for such $\fp$,
$|\Delta_{\fp}|=|c_\fp|^2\cdot |\Delta_{F(\pi_\fp)}| > \kappa$.
However,   \cite[Thm. 1.1]{GaPa18} does not imply that $|\Delta_{\fp}|$ has to grow with $\deg p$. 
In fact, computationally, Garai and the second author have found that it happens that $c_{\fp}=1$; in this case,
 \cite[Thm. 1.1]{GaPa18} does not give any lower bound on $|\Delta_{\fp}|$. 
On the other hand,
Theorem \ref{thm1} does not imply that 
 we can find any $\fp$ such that $|c_{\fp}|>\kappa$.  
Thus these two results are complementary to each other. 
\end{remark}

\begin{remark}\label{remark-CM-case}
	If $\End_{F^\alg}(\psi)\neq A$, then $\End_{F^\alg}(\psi)=\cO$ is an order in an 
	imaginary quadratic extension $K$ of $F$,
	in which case the growth of $|\Delta_{\fp}|$ is vastly different from that shown in Theorem \ref{thm1}.
	On one hand, if $\fp\in \cP(\psi)$  splits in $K$, then $\cO\subseteq \End_{\F_\fp}(\psi\otimes \F_\fp)\subseteq \cO_K$  (see \cite[Lem. 3.3]{Ge83}),
	which
	implies  that $|\Delta_{\fp}|\leq |\Delta_{\cO}|$, where $\Delta_{\cO}$ is the discriminant of $\cO$. 
	Hence $|\Delta_{\fp}|$ remains bounded as $\fp$ varies over the primes that split in $K$. In particular, Theorem \ref{thm1} is false without its assumption. 
	On the other hand, if $\fp \in \cP(\psi)$ is inert in $K$, then $\psi\otimes \F_\fp$ is supersingular,
	which implies that
	 $A[\pi_\fp]=A[\sqrt{\alpha p}]$ for some $\alpha \in \F_q^\times$, 
	 and that $A[\pi_\fp]=\cE_\fp=\cO_{F(\pi_\fp)}$
	 (see Lemma 5.2 and Theorem 5.3 of \cite{Ge83}).
	Hence $|\Delta_{\fp}|=|p|$, a much larger growth than that shown in  Theorem \ref{thm1}.
	One can also prove that, in this case,
	 $d_{1, \fp}=1$ and $d_{2, \fp}=p-\beta$ for some $\beta \in \F_q^\times$ 
	(see \cite[Cor. 3]{CoPa15}). Hence $|d_{2, \fp}|=|p|$, which  is as large as possible. 
\end{remark}

\begin{remark}\label{remark-upperbound} 
From the theory of Drinfeld modules over finite fields, one can deduce that the 
	discriminant of $A[\pi_\fp]$ has degree $\leq \deg p$. 
	More precisely, the characteristic polynomial of $\pi_{\fp}$ is of the form 
	$X^2+a_{\fp}X+\mu_{\fp}  p \in A[X]$,
where $\mu_{\fp} \in \F_q^\times$
and
$\deg a_{\fp} \leq \frac{\deg p}{2}$.
This implies that $|\Delta_{\fp}| \leq \left|a_{\fp} - 4 \mu_{\fp} p \right| \leq |p|$.  
Note that the coefficients $a_{\fp}$, $\mu_{\fp}$ also depend on $\psi$, although this will not be explicitly indicated in our notation.
	\end{remark}

The lower bound on $|\Delta_{\fp}|$  in Theorem \ref{thm1} holds for \textit{all} 
 primes $\fp\lhd A$, with finitely many exceptions. The next theorem gives a stronger lower bound,
 almost as close as the upper bound of Remark \ref{remark-upperbound}, 
 which holds  for {\it{a set}} of primes $\fp$ {\it{of Dirichlet  density 1}}:

\begin{theorem}\label{thm3}
Assume $\End_{F^\alg}(\psi)=A$. 
For any function $f: (0, \infty) \longrightarrow (0, \infty)$ with $\ds\lim_{x \rightarrow \infty} f(x) = \infty$,
we have that, as $x \rightarrow \infty$, 
$$
\#\left\{
\mathfrak{p} \in  \cP(\psi):
\deg p = x,
 |\Delta_{\fp}|
 >
\frac{\left|a_{\fp} - 4 \mu_{\fp} p \right| }{q^{f(\deg p)}} 
\right\}
\sim
\pi_F(x),
$$
where
$\pi_F(x) := \#\left\{\fp \lhd A \ : \deg p = x\right\}$.
Moreover, 
the Dirichlet density of the set
$$
\left\{
\fp \in  \cP(\psi):
 |\Delta_{\fp}|
 >
\frac{\left|a_{\fp} - 4 \mu_{\fp} p \right| }{q^{f(\deg p)}} 
\right\}
$$
exists and equals 1.
\end{theorem}

Theorem \ref{thm3} is a Drinfeld module (and unconditional)  analogue
of a recent result of the first author and Fitzpatrick for 
elliptic curves over $\Q$; see \cite{CoFi19}.
It is inspired by the results of \cite{CoSh15}
and
relies on the main ideas of \cite{CoPa15}.

In Sections \ref{sOrders}, \ref{sQF}, \ref{sjInv} 
we review and prove
several results
about orders, quadratic forms, and the $j$-invariants of Drinfeld modules,
as needed in the proofs of the main theorems,
which we give in Sections 5, 6.

\medskip

{\bf{Notation}}. 
Throughout the paper,  we use the standard $\sim$, $\o$, $\O$,  $\ll$, $\gg$ notation, which we now recall:
given suitably defined real functions $h_1, h_2$,
we say that
$h_1 \sim h_2$ if $\ds\lim_{x \rightarrow \infty} h_1(x)/h_2(x) = 1$;
we say that
$h_1 = \o(h_2)$ if $\ds\lim_{x \rightarrow \infty} h_1(x)/h_2(x) = 0$;
we say that
$h_1 = \O(h_2)$ or $h_1 \ll h_2$ or $h_2 \gg h_1$
if 
$h_2$ is positive valued 
and
 there exists a positive constant $C$ such that 
$|h_1(x)| \leq C h_2(x)$ for all $x$ in the domain of $h_1$;
we say that
$h_1 = \O_D(h_2)$ or $h_1 \ll_D h_2$ or $h_2 \gg_D h_1$
if 
$h_1 \ll h_2$
and
the implied $\O$-constant $C$ depends on priorly given data  $D$.
We make the convention that any implied $\O$-constant may depend on $q$ without any explicit specification.

\section{$A$-orders}\label{sOrders}

Let $K/F$ be a quadratic imaginary extension. 
Let $B$ be the integral closure of $A$ in $K$. An \textit{$A$-order} in $K$ is an $A$-subalgebra $\cO$ 
of $B$ with the same unity element and such that $B/\cO$ has finite cardinality. 
Note that an $A$-order $\cO$ 
is a free $A$-module of rank $2$
and  that
there is an $A$-module isomorphism $B/\cO\cong A/cA$ for a unique nonzero monic polynomial $c\in A$, called 
the \textit{conductor} of $\cO$ in $B$. It is easy to show that $\cO=A+cB$. 

Let $\{\alpha_1, \alpha_2\}$ be a basis 
of $\cO$ as a free $A$-module, and $\sigma\in \Gal(K/F)\cong \Z/2\Z$ be the generator of $\Gal(K/F)$. 
The discriminant of $\{\alpha_1, \alpha_2\}$ 
is $$\disc(\alpha_1, \alpha_2)=\det\begin{pmatrix} \alpha_1 & \sigma(\alpha_1)\\ \alpha_2 & \sigma(\alpha_2)\end{pmatrix}^2.$$
If $\{\beta_1, \beta_2\}$ is another $A$-basis of $\cO$, then $\disc(\beta_1, \beta_2)=\kappa^2\cdot \disc(\alpha_1, \alpha_2)$ for some 
$\kappa\in \F_q^\times$. The \textit{discriminant} $\Delta_{\cO}$ of $\cO$ is defined to be
$\disc(\alpha_1, \alpha_2)$, up to an $(\F_q^\times)^2$-multiple. 
It is elementary to show that $\cO=A[\sqrt{\Delta_{\cO}}]$ and $(\Delta_{\cO})=c^2\cdot (\Delta_B)$, where 
$(a)$ denotes the ideal generated by $a\in A$. Note that $\Delta_B$ is square-free. 

\begin{remark}\label{rem-inf}
	The splitting behavior of $\infty=1/T$ in any quadratic extension $K/F$ can be described using the discriminant $\Delta_B$. Namely, 
	$\infty$ ramifies in $K/F$ $\Longleftrightarrow$ $\deg \Delta_B$ is odd;  
	$\infty$ splits in $K/F$ $\Longleftrightarrow$ $\deg \Delta_B$ is even and the leading coefficient of $\Delta_B$ is 
	a square in $\F_q^\times$; 
	$\infty$ is inert in $K/F$ $\Longleftrightarrow$ $\deg \Delta_B$ is even and 
	the leading coefficient of $\Delta_B$ is not a square in $\F_q^\times$.
\end{remark}

There are two important groups associated to an $A$-order ${\cal{O}}$:
the {\it{unit group}} $\cal{O}^{\times}$ 
of invertible elements of $\cal{O}$
and
the {\it{ideal class group}} $\Cl(\cO)$
of  classes of proper (invertible) fractional ideals of ${\cal{O}}$; cf. \cite[p. 136]{Co89}.  
The ideal class group is finite and its cardinality
$h(\cal{O})$ is called the {\it{class number}} of $\cal{O}$. 
The class numbers of $\cO$ and $B$ are related by the following well-known formula 
(cf. \cite[Thm. 7.24]{Co89} or \cite[pp. 323--324]{Yu95bis}):
\begin{equation}\label{h-O-h-O-K}
h(\cO)=h\left(B\right)\frac{|c|}{\left[B^\times : \cO^\times\right]}
\ds\prod_{\substack{\ell \mid c\\ \ell \text{ monic irreducible}}
}
\left(1 - \left(\frac{K}{\ell}\right)\frac{1}{ \left|\ell\right| }\right),
\end{equation}
where 
$$
\left(\frac{ K}{\ell}\right) :=
\begin{cases}
1 & \text{if $(\ell)$ splits in $K$}, \\ 
-1 & \text{if $(\ell)$ is inert in $K$}, \\ 
0 & \text{if $(\ell)$ ramifies in $K$}. 
\end{cases}
$$
Hence 
\begin{equation}\label{eqh-O-B1}
h(\cO)\leq h\left(B\right)\cdot |c| \prod_{\substack{\ell \mid c\\ \ell \text{ monic irreducible}}}
\left(1 +\frac{1}{ \left|\ell\right| }\right). 
\end{equation}

The product on the right-hand side of (\ref{eqh-O-B1}) may be bounded from above as follows: 
\begin{lemma}\label{lemh-O-B2}
	\begin{equation}
	\prod_{\substack{\ell \mid c\\ \ell \mathrm{\ monic\ irreducible}}}
	\left(1 +\frac{1}{ \left|\ell\right| }\right) \ll  \log\log|c|.
	\end{equation}
\end{lemma}
\begin{proof} To simplify the notation, in the proof we assume that $\ell$ is always monic and irreducible. 
	Note that 
	$$
	\prod_{\ell\mid c}
	\left(1 +\frac{1}{ \left|\ell\right| }\right) < \prod_{\ell\mid c} \frac{|\ell|}{|\ell|-1}=\frac{|c|}{\varphi_A(c)},
	$$
	where $\varphi_A(c)=|c|\prod_{\ell\mid c}(1-1/|\ell|)$ is the analogue of Euler's function,
	and recall that  the well-known bound 
	$$|c|/\varphi_A(c) \ll \log\log|c|$$
	for the classical Euler function is valid also for its function field analogue; see \cite[Lem. 2.2]{Br}. 
	Then Lemma \ref{lemh-O-B2} follows.
	\end{proof}

Next, the class number $h(B)$ of the  maximal order may be estimated from above as follows:

\begin{lemma}\label{lemh-O-B3-bis}
$$
	h(B)
	\leq
\begin{cases}
	\frac{\sqrt{|\Delta_B|} \deg \Delta_B}{\sqrt{q}},
	& \text{if $\deg \Delta_B$ is odd,}
	\\
	\\
	\frac{2 \sqrt{|\Delta_B|} \deg \Delta_B}{q + 1},
	&  \text{otherwise.}
	\end{cases}
	$$	
	\end{lemma}
\begin{proof} 
	The quadratic symbol $\left(\frac{ K}{\ell}\right)$
	gives rise to a quadratic character $\chi_{\Delta_B}(\cdot)$ and 
	to an associated $L$-function $L(s, \chi_{\Delta_B})$; see \cite[p. 316]{Ro02}. 
	On one hand, by the function field analogue of the classical class number formula (see \cite[Thm. 17.8A]{Ro02}), 
	$$
	L(1, \chi_{\Delta_B})
	=
	\begin{cases}
	\frac{\sqrt{q}}{\sqrt{|\Delta_B|}} h(B), & \text{if $\deg \Delta_B$ is odd,}
	\\
	\frac{q+1}{2 \sqrt{|\Delta_B|}} h(B), & \text{otherwise.}
		\end{cases}
	$$
	(Here we use our assumption that $K/F$ is imaginary.) On the other hand, by \cite[Lem. 17.10]{Ro02}, 
	$$
	L(1, \chi_{\Delta_B})= \sum_{\substack{0\neq m\in A\\ m \text{ monic}\\ \deg m<\deg\Delta_B}} \frac{\chi_{\Delta_B}(m)}{|m|}. 
	$$
	Hence 
	$$
	\left| 	L(1, \chi_{\Delta_B}) \right|\leq  \sum_{\substack{0\neq m\in A\\ m \text{ monic}\\ \deg m<\deg\Delta_B}} \frac{1}{|m|} 
	=\sum_{d=0}^{\deg(\Delta_B)-1} q^{-d}\sum_{\substack{0\neq m\in A\\ m \text{ monic}\\ \deg m=d}}1 = \deg \Delta_B, 
	$$
	where $\left| 	L(1, \chi_{\Delta_B}) \right|$ denotes the usual absolute value on $\C$. 
	Combining this bound with the class number formula, we obtain the stated bound for $h(B)$.
	\end{proof}

Putting together \eqref{eqh-O-B1}, Lemma \ref{lemh-O-B2} and Lemma \ref{lemh-O-B3-bis}, we obtain 
an upper bound for the class number $h(\cO)$ of the arbitrary order $\cO$:

\begin{equation}\label{bound-order-class-number}
h(\cO)\ll \sqrt{|\Delta_{\cO}|}\cdot (\deg \Delta_{\cO})^2. 
\end{equation}

\section{Quadratic Forms}\label{sQF}

Let $f(x, y)=ax^2+bxy+cy^2\in A[x, y]$ be a quadratic form. The \textit{discriminant} of $f(x, y)$ is $b^2-4ac$. 
The quadratic form $f(x, y)$ is \textit{primitive} if $\gcd(a, b, c)=1$. The group $\GL_2(A)$ acts 
on the set of primitive quadratic forms as usual: if $\begin{pmatrix} a & b/2 \\ b/2 & c\end{pmatrix}$ 
is the matrix of $f(x, y)=\begin{pmatrix} x & y\end{pmatrix} \begin{pmatrix} a & b/2 \\ b/2 & c\end{pmatrix} \begin{pmatrix} x \\ y\end{pmatrix}$ 
and ${\mathcal{M}} \in \GL_2(A)$, 
then ${\mathcal{M}}^t \begin{pmatrix} a & b/2 \\ b/2 & c\end{pmatrix}  {\mathcal{M}}$ is the matrix of 
${\mathcal{M}} f$. 
Two primitive quadratic forms $f$ and $g$ are \textit{properly equivalent} if 
$g= {\mathcal{M}} f$ for some 
${\mathcal{M} }\in \SL_2(A)$. 
In the proof of Theorem \ref{thm1} we will need the following analogue of a well-known classical result: 

\begin{theorem}\label{thm-ClO-QF}
Let $\cO$ be an imaginary quadratic $A$-order, of discriminant $\Delta_{\cO}$.
	If $ax^2+bxy+cy^2\in A[x, y]$ is a primitive quadratic form of discriminant $\Delta_\cO$, then 
	$A+\frac{-b+\sqrt{\Delta_\cO}}{2a}A$ is a proper fractional ideal of $\cO$. The map 
	$$
	ax^2+bxy+cy^2 \longmapsto A+\frac{-b+\sqrt{\Delta_\cO}}{2a}A
	$$
	induces a bijection between proper equivalence classes of  primitive quadratic forms of discriminant $\Delta_\cO$ 
	and $\Cl(\cO)$. 
\end{theorem} 
\begin{proof}
	The proof of Theorem 7.7 in \cite{Co89} works also in this context; see \cite{Be09}  for the details. 
\end{proof}

\begin{lemma}\label{lem-reduced}
	Every primitive quadratic form over $A$ is properly equivalent to a quadratic form $ax^2+bxy+cy^2$ 
	such that 
	\begin{equation}\label{reduced}
	\deg b < \deg a \leq \deg c.
	\end{equation}
\end{lemma}
\begin{proof} Among all forms properly equivalent to the given one, pick 
$f(x, y)=ax^2+bxy+cy^2$ so that $\deg b$ is as small as possible  
(note that $b$ can be zero, in which case, by our convention, $\deg 0=-\infty$). 

If  $\deg a \leq \deg b$, then 
$f$ is properly equivalent to 
	$$g(x, y)=f(x+my, y)=ax^2+(2am+b)xy+c'y^2$$
	for some $m\in A$. Using the division algorithm in $A$, we can choose $m$ so that $\deg(2am+b)<\deg a$, 
	which contradicts our choice of $f(x, y)$. Thus, $\deg b < \deg a$, and the inequality $\deg b<\deg c$ follows similarly. 
	
	If $\deg a>\deg c$, we need to interchange the outer coefficients, which is accomplished by the 
	proper equivalence $(x, y)\mapsto (-y, x)$. The resulting form satisfies $\deg b < \deg a \leq \deg c$. 
\end{proof}

Let $\Fi$ be the completion of $F$ with respect to the absolute value $|\cdot|$, and $\Ci$ be the completion of 
the algebraic closure $\Fi^\alg$. 
We use the same notation for the unique extension of $|\cdot|$ to $\Ci$. The \textit{imaginary part} of $z\in \Ci$ is 
$$
|z|_i:=\min_{x\in \Fi} |z-x|. 
$$
Obviously, $|z|_i\leq |z|$. 

\begin{lemma}\label{disc-reduced-form}
	Let $f(x, y)=ax^2+bxy+cy^2$ be a primitive quadratic form over $A$ with discriminant $\Delta$. Assume $F(\sqrt{\Delta})$ 
	is imaginary and $f(x, y)$ satisfies \eqref{reduced}. Then 
	\begin{equation}\label{eq-FundDom}
	1\leq \left|\frac{-b+\sqrt{\Delta}}{2a}\right|_i = \left|\frac{-b+\sqrt{\Delta}}{2a}\right|\leq |\sqrt{\Delta}|. 
	\end{equation}
\end{lemma}
\begin{proof}
	We have $\Delta=b^2-4ac$. If $\deg b <\deg a \leq \deg c$,  then 
	$$
	\deg \Delta =\deg(b^2-4ac)=\deg(ac)\geq 2\deg a. 
	$$
	Hence $|\sqrt{\Delta}|\geq |a|>|b|$. By the strong triangle inequality, we get 
	\begin{equation}\label{eq-deg-disc}
	\left|\frac{-b+\sqrt{\Delta}}{2a}\right|=\frac{|-b+\sqrt{\Delta}|}{|a|}=\frac{|\sqrt{\Delta}|}{|a|}. 
	\end{equation}
	Since $|a|\geq 1$, we have 
	$$
	1\leq \frac{|\sqrt{\Delta}|}{|a|}\leq |\sqrt{\Delta}|. 
	$$
	This proves the outer two inequalities of \eqref{eq-FundDom}. 
	
	It remains to prove the middle equality.  From \eqref{eq-deg-disc} we get 
	$$
	\log_q \left|\frac{-b+\sqrt{\Delta}}{2a}\right| = \log_q \frac{|\sqrt{\Delta}|}{|a|} = \frac{\deg \Delta}{2}-\deg a. 
	$$
	First suppose $\deg \Delta$ is odd. Then $\log_q \left|\frac{-b+\sqrt{\Delta}}{2a}\right| \not\in \Z$. 
	In this case the desired middle equality follows from the more general fact that 
	if $z\in \Ci$ is such that $\log_q |z| \not\in \Z$, then $|z|= |z|_i$, as we now explain. 
	On one hand, 
	if  $\log_q |z| \not\in \Z$, then for any $x \in F_\infty$
	we have
	$|z| \neq |x|$. 
	As such, the strong triangle inequality  implies
	$|z-x| = \max\left\{|z|, |x|\right\}\geq |z|$, showing that $|z|_i\geq |z|$. On the other hand, $|z|_i\leq |z|$.
	Thus 
	we must have $|z|_i=|z|$. 
	
	Next suppose $u:=\deg \Delta$ is even. Since $F(\sqrt{\Delta})$ is assumed to be imaginary, the leading 
	coefficient of $\Delta$ is not a square in $\F_q^\times$ (see Remark \ref{rem-inf}) and 
	$\Fi(\sqrt{\Delta})=\F_{q^2}\Fi$.  
	Choosing $1/T$ as the uniformizer of $\F_{q^2}\Fi$, we can expand 
	$\sqrt{\Delta}=\alpha \left(\frac{1}{T}\right)^{-u/2}+\text{higher degree terms in }1/T$, where $\alpha\in \F_{q^2}-\F_q$. 
	Since $u/2\geq \deg a>\deg b$, the $1/T$-expansion of $(-b+\sqrt{\Delta})/2a$ is 
	$$
	\frac{-b+\sqrt{\Delta}}{2a} = \beta \left(\frac{1}{T}\right)^{-v}+\text{higher degree terms in }1/T, 
	$$
where $v:=\frac{u}{2}-\deg a\geq 0$ and $\beta\in \F_{q^2}-\F_q$. If $\left|\frac{-b+\sqrt{\Delta}}{2a}-x\right|< \left|\frac{-b+\sqrt{\Delta}}{2a}\right|$ 
for some $x\in \Fi$, then the $1/T$-expansion of $x$ must have the form $\beta \left(\frac{1}{T}\right)^{-v}+\text{higher degree terms in }1/T$. But this is 
not possible since $\beta\not\in \F_q$. Therefore, $\left|\frac{-b+\sqrt{\Delta}}{2a}-x\right|\geq \left|\frac{-b+\sqrt{\Delta}}{2a}\right|$ 
for all $x\in \Fi$, which implies $\left|\frac{-b+\sqrt{\Delta}}{2a}\right|_i=\left|\frac{-b+\sqrt{\Delta}}{2a}\right|$. 
\end{proof}

\section{The $j$-invariant of a rank 2 Drinfeld module}\label{sjInv}

Let $\gamma: A\longrightarrow L$ be an $A$-field and  let $\psi: A \longrightarrow L\{\tau\}$ be a  Drinfeld module over $L$ of rank $2$, 
defined by $\psi_T = \gamma(T) + g_1\tau + g_2 \tau^2$
for some $g_1, g_2 \in L$ with $g_2 \neq 0$. Two Drinfeld modules $\psi$ and $\phi$ are said to be isomorphic over an extension $L'$ of $L$ 
if $\psi_T=c^{-1}\phi_T c$ for some $c\in L'$. 
The quantity
\begin{equation}\label{def-j-invariant}
j(\psi) := \frac{g_1^{q+1}}{g_2} \in L
\end{equation}
is called the {\emph{$j$-invariant of $\psi$}}. It is easy to show that two Drinfeld modules $\phi$ and $\psi$ of rank $2$  
are isomorphic over $L^\alg$ if and only if $j(\phi)=j(\psi)$. 

Now assume $L =\Ci$. Let $\Omega:=\Ci-\Fi$ be the Drinfeld half-plane. 
The group $\GL_2(A)$ acts on $\Omega$ by linear fractional transformations 
$$
\begin{pmatrix}a & b \\ c & d
\end{pmatrix}z
:=\frac{a z + b}{c z + d}. 
$$
The set 
\begin{equation}\label{eq-FundDomain}
\cF:=\{z\in \Omega\ :\ |z|=|z|_i\geq 1 \}
\end{equation}
is as close as possible to a ``fundamental domain'' for the action of $\GL_2(A)$ on $\Omega$; 
see \cite[Prop. 6.5]{Ge97}. 
In particular, every element of $\Omega$ is $\GL_2(A)$-equivalent to some element of $\cF$. 

To each $z\in \Omega$, we associate the lattice $A+Az\subset \Ci$. By the analytic theory of Drinfeld 
modules, the lattice $A+Az$ corresponds to a Drinfeld module $\psi^z$ of rank $2$ defined over $\Ci$. Moreover, 
it can be shown  that $\psi^z\cong \psi^{z'}$ (over $\Ci$)  if and only if $z=\gamma z'$ for some $\gamma\in \GL_2(A)$; see \cite{Dr74}. 
Therefore, 
the map $z\longmapsto \psi^z$ induces a bijection between the orbits $\GL_2(A)\setminus \Omega$ and 
the isomorphism classes of rank $2$ Drinfeld modules over $\Ci$. Thanks to these properties,
there exists a $\GL_2(A)$-invariant function
\begin{equation}\label{j-function}
j : \Omega \longrightarrow \C_{\infty}, \quad j(z) := j\left(\psi^z\right). 
\end{equation}

\begin{theorem}\label{thm-gekeler-j-growth}
	For $z\in \cF$, we have $\log_q|j(z)|< q^2\cdot |z|$. 
\end{theorem}
\begin{proof}
	See Theorem 6.6 in \cite{Ge97}. 
\end{proof}

A Drinfeld module $\psi$ of rank $2$ over $\Ci$ is said to have \textit{complex multiplication} if $\End_{\Ci}(\psi)\neq A$. 
\begin{theorem}\label{thm-CM}
	Suppose that the Drinfeld module $\psi:=\psi^z$  defined by some $z\in \Omega$ has complex multiplication. Then
	the following properties hold.
	\begin{enumerate}
		\item[(i)] $K:=F(z)$ is an imaginary quadratic extension of $F$. 
		\item[(ii)] $\cO := \End_{\Ci}(\psi)$ is an $A$-order in $K$.
		\item[(iii)] $K(j(\psi))/K$ is a finite abelian extension. 
		\item[(iv)] $j(\psi)$ is integral over $A$. 
		\item[(v)] $\Gal(K(j(\psi))/K) \cong \Cl(\cO)$.
		\item [(vi)]
		$
		\left\{\sigma(j(\psi))\mid \sigma\in \Gal(K(j(\psi))/K) \right\}= 
		\left\{j\left(\frac{-b+\sqrt{\Delta_\cO}}{2a}\right)\ \bigg|\ [ax^2+bxy+cy^2]_{\Delta_\cO}/\SL_2(A)\right\},
		$
		
		\noindent
		where $[ax^2+bxy+cy^2]_{\Delta_\cO}/\SL_2(A)$ denotes the proper equivalence class of the primitive 
		quadratic form $ax^2+bxy+cy^2$ of discriminant $\Delta_\cO$. 
	\end{enumerate}
\end{theorem}
\begin{proof}
	For (i)-(v), see Section 4 in \cite{Ge83}. 
	For (vi), proceed as follows.
	On one hand, by \cite[Cor. 4.5]{Ge83}, the set of 
	Galois conjugates of $j(\psi)$ is equal to the set $\{j(z')\}$, where $A+Az'$ runs over the equivalence 
	classes of  proper fractional ideals of $\cO$. On the other hand, Theorem \ref{thm-ClO-QF} gives explicit expressions 
	for representatives of these ideal classes in terms of the equivalence classes of quadratic forms. 
	This completes the proof.
\end{proof}

\section{Proof of Theorems \ref{thm1} and \ref{thm2}}

Let $\psi$ be a Drinfeld module of rank $2$ over $F$ 
and 
 let $\fp\lhd A$ be a fixed prime where $\psi$ has good reduction. Let $\psi\otimes \F_\fp$ be the 
reduction of $\psi$ at $\fp$. As we mentioned in the introduction, $\cE_{\psi, \fp}:=\End_{\F_\fp}(\psi\otimes \F_\fp)$ 
is an $A$-order in the imaginary quadratic extension $F(\pi_\fp)$ of $F$. 
Since $\fp$ remains  fixed in this section, for simplicity of notation in the proofs below we write 
$\cE := \cE_{\psi, \fp}$, $\Delta := \Delta_{\fp}$,
 and $K:=F(\pi_\fp)$.

\subsection{Proof of Theorem \ref{thm1}}
\begin{proposition}\label{prop-CMlift}
	There exists a Drinfeld module $\Psi$ of rank $2$ over $\Ci$ for which the following properties hold.
	\begin{enumerate}
		\item[(i)] $\End_{\Ci}(\Psi)=\cE$. 
		\item[(ii)] There exists a prime $\fP$ of $K(j(\Psi))$ lying over $\fp$ such that $j(\Psi)\equiv j(\psi)\mod \fP$. 
	\end{enumerate}
\end{proposition}
\begin{proof} 
	Note that for (ii) we implicitly use Theorem \ref{thm-CM}: assuming (i), by Theorem \ref{thm-CM}, $j(\Psi)$ 
	is algebraic over $A$, so $K(j(\psi))$ is a finite algebraic extension of $F$ and $j(\Psi)\mod \fP$ makes sense. 
	
	Since the rank of $\psi\otimes \F_\fp$ is $2$, by \cite[Prop. 24]{CoPa15}, the field $K$ is ``good" for $\psi\otimes \F_\fp$ 
	in the sense of \cite{CoPa15}. Then by Theorem 22 and its proof in \cite{CoPa15}, 
	there exists a discrete valuation ring $R$ with maximal ideal $\cM$,  
	equipped with an injective homomorphism $\gamma: A\to R$, and having the properties:
	\begin{enumerate}
		\item[(a)] $\cM\cap A=\fp$ and $R/\cM\cong A/\fp$;
		\item[(b)] there exists a Drinfeld module $\Psi: A\to R\{\tau\}$ of rank $2$ such that $\cE\subseteq \End_R(\Psi)$;
		\item[(c)] $\Psi\otimes\F_\fp:=\Psi \mod \cM$ is isomorphic to $\psi\otimes\F_\fp$ over $\F_\fp$.  
	\end{enumerate}
	It is not hard to deduce from \cite[Lem. 3.3]{Ge83} that under reduction modulo $\cM$ we get an injection 
$\End_{L}(\Psi)\hookrightarrow \End_{\F_\fp}(\Psi\otimes\F_\fp)$, where $L$ is the fraction field of $R$. 
Hence $\cE\subseteq \End_{L}(\Psi) \subseteq \End_{\F_\fp}(\Psi\otimes\F_\fp)=\cE$, which implies that 
$\End_L(\Psi)=\cE$. 
By considering the action of $\cE$ on the tangent space of $\Psi$, one deduces that $K$ is a subfield of $L$. 
Thus $K(j(\Psi))$ is a subfield of $L$. Let $\fP:=\cM\cap K(j(\Psi))$. 
Since $A/\fp\subseteq \cO_{K(j(\Psi))}/\fP\subseteq R/\cM$, 
$\fP$ is a maximal ideal of the integral closure $\cO_{K(j(\Psi))}$ of $A$ in $K(j(\Psi))$, with residue field $\F_\fp$. 
From the construction, it is clear that $\Psi\mod \fP$ is isomorphic to $\psi\otimes\F_\fp$ over $\F_\fp$. 
In particular, $j(\Psi)\equiv j(\psi)\mod \fP$. 

Finally, we can embed $L$ into $\Ci$ and consider $\Psi$ as a Drinfeld module over $\Ci$. Since $\End_{\Ci}(\Psi)/\End_L(\Psi)$ is a free $A$-module 
and $\rank_A\End_{\Ci}(\Psi)\leq 2$, we conclude that $\End_{\Ci}(\Psi)=\cE$.
\end{proof}

Now assume that $\End_{F^\alg}(\psi)=A$. Let $\Psi$ be a Drinfeld module over $\Ci$ as in Proposition \ref{prop-CMlift}. 
We have $j(\psi)\neq j(\Psi)$ (as elements of $\Ci$), since otherwise 
$\psi\cong \Psi$ over $\Ci$, which would imply $\End_{F^\alg}(\psi)=\End_{\Ci}(\psi)=\cE$, a contradiction. 
Write $j(\psi)=n/m$ 
with relatively prime $n, m\in A$. Let $\fP$ be the prime of $K(j(\Psi))$ from Proposition \ref{prop-CMlift}. Then $j(\Psi)\equiv \frac{n}{m}\mod \fP$ 
and $j(\Psi)\neq \frac{n}{m}$
imply $0\neq n-m\cdot j(\Psi)\in \fP$. By Theorem \ref{thm-CM}, $K(j(\Psi))/K$ is an abelian extension and 
\begin{align*}
\mathrm{Nr}_{K(j(\Psi))/K}(n-m\cdot j(\Psi)) & =\prod_{\sigma\in \Gal(K(j(\Psi))/K)}(n-m\cdot\sigma(j(\Psi))) \\
&=\prod_{[ax^2+bxy+cy^2]_{\Delta}/\SL_2(A)} \left(n-m\cdot  j\left(\frac{-b+\sqrt{\Delta}}{2a}\right)\right). 
\end{align*}

On one hand,
since $n-m\cdot j(\Psi)\in \fP$, we have $\alpha:=\mathrm{Nr}_{K(j(\Psi))/K}(n-m\cdot j(\Psi))\in \fp'$, where $\fp'$ 
is the prime of $K$ lying under $\fP$. Letting $\alpha'$ be the conjugate of $\alpha$ over $F$,
we obtain that  $\alpha\alpha'\in \fp$,  which gives
 $|\alpha|\geq |p|^{1/2}$. 
 
 On the other hand, by the strong triangle inequality,  
$$
|\alpha|=\prod_{[ax^2+bxy+cy^2]_{\Delta}/\SL_2(A)} \max \left\{|n|, |m|\cdot  \left|j\left(\frac{-b+\sqrt{\Delta}}{2a}\right)\right|\right\}. 
$$

Remark that,
by Lemma \ref{lem-reduced}, we can assume that the triplets $(a, b, c)$ above satisfy \eqref{reduced}. 
Under this assumption, Lemma \ref{disc-reduced-form} implies that $\frac{-b+\sqrt{\Delta}}{2a}\in \cF$, with 
$\cF$ defined in \eqref{eq-FundDomain}. Then Theorem \ref{thm-gekeler-j-growth} and Lemma \ref{disc-reduced-form} imply 
$$
\log_q\left|j\left(\frac{-b+\sqrt{\Delta}}{2a}\right)\right|\leq q^2 \left|\frac{-b+\sqrt{\Delta}}{2a}\right|\leq q^2 \left|\sqrt{\Delta}\right|.
$$

Combining our estimates, we get 
$$
|p|^{\frac{1}{2}}\leq \prod_{[ax^2+bxy+cy^2]_{\Delta}/\SL_2(A)} \max \left\{|n|, |m|\cdot q^{q^2 \left|\sqrt{\Delta}\right|}\right\} 
= \max \left\{|n|, |m|\cdot q^{q^2 \left|\sqrt{\Delta}\right|}\right\}^{h(\cE)}. 
$$
Since $n$ and $m$ are determined by $\psi$, we deduce that
$$
\deg p \ll_{\psi} h(\cE)\cdot \left|\sqrt{\Delta}\right|. 
$$ 
Furthermore, since, by \eqref{bound-order-class-number}, 
$h(\cE)\ll \sqrt{|\Delta|}\cdot (\deg\Delta)^2$, we deduce that
\begin{equation}\label{eqThm1Ver1}
\deg p \ll_{\psi} |\Delta|\cdot (\deg\Delta)^2. 
\end{equation}

We claim that \eqref{eqThm1Ver1} implies 
\begin{equation}\label{eqThm1Ver2}
|\Delta|\gg_{\psi} \frac{\log |p|}{(\log\log |p|)^2}.
\end{equation}
If $|\Delta|\gg \log |p|$, then \eqref{eqThm1Ver2} is immediate. Now assume $|\Delta|\ll \log |p|$, so that 
$\log |\Delta|\ll \log \log |p|$. Then from \eqref{eqThm1Ver1}, we get 
$$|\Delta|\gg_{\psi}  \frac{\log |p|}{(\log|\Delta|)^2}\gg_{\psi} \frac{\log |p|}{(\log \log |p|)^2}. 
$$
This completes the proof of Theorem \ref{thm1}.

\subsection{Proof of Theorem \ref{thm2}} 

We start by recalling more general facts.
Let $\gamma: A\to L$ be an $A$-field and let $\phi$ be a Drinfeld module of rank $r$ over $L$. 
Then $\phi$ endows any field 
extension $L'$ of $L$ with an $A$-module structure, where $m\in A$ acts by $\phi_m$. More precisely, if 
$\phi_m=\gamma(m)+\sum_{1 \leq i  \leq r \deg m}g_i(m)\tau^i$,
that is, 
$\phi_m(x)=\gamma(m)x+\sum_{1 \leq i \leq r \deg m}g_i(m) x^{q^i}\in L[x]$,
then for $\lambda \in L'$ define
$m\circ \lambda:=\phi_m(\lambda)$. 
The \textit{$m$-torsion} $\phi[m]\subset L^\alg$ of $\phi$ 
is the set of zeros of the polynomial $\phi_m(x)$.
 It is clear that $\phi[m]$ has a natural structure of an $A$-module
 and
it is not hard to show that, as $A$-modules,  $\phi[m]\subseteq (A/mA)^{\oplus r}$,
 with an equality if and only if $m$ is relatively prime to $\ker(\gamma)$;  
see \cite{Go96}. 

We return to the Drinfeld module $\psi$ of rank $2$ over $F$ and,  
as in the introduction,
we consider $\F_\fp$ as an $A$-module via the action of $\psi\otimes \F_\fp$ and 
denote it by ${^\psi}\F_\fp$; then
$$
{^\psi}\F_\fp\cong A/d_{1, \fp}A\times A/d_{2,  \fp} A
$$
for uniquely determined nonzero monic polynomials $d_{1, \fp}, d_{2,  \fp}\in A$ such that 
$d_{1,  \fp} \mid d_{2, \fp}$. (Note that there are at most two terms 
because ${^\psi}\F_\fp$ is a finite $A$-module, so for some $d\in A$ we have ${^\psi}\F_\fp\subseteq (\psi\otimes \F_\fp)[d]\subseteq A/dA\times A/dA$.)
Since $\fp$ remains  fixed in this section, for simplicity of notation in this proof we will write 
$d_1 := d_{1,  \fp}$, $d_2 := d_{2, \fp}$.

Suppose $d_1\neq 1$. Then $(\psi\otimes \F_\fp)[d_1]$ is rational over $\F_\fp$, i.e., all roots of $\psi_{d_1}(x)$ are in $\F_\fp$, 
and $p\nmid d_1$. Let $\pi_{\fp}=\tau^{\deg(p)}\in \cE$ be the Frobenius endomorphism of $\psi\otimes \F_\fp$. The fact that 
$(\psi\otimes \F_\fp)[d_1]$ is rational over $\F_\fp$ implies that $\pi_{\fp}=1+d_1\alpha$ for some $\alpha\in \cE$; see the proof of Theorem 1.2 in \cite{GaPa18}.  
From the theory of Drinfeld modules over finite fields (see \cite{Yu95}, \cite{Ge91}) one knows that the minimal polynomial of $\pi_{\fp}$ 
over $A$ has the form 
\begin{equation}\label{char-poly}
P_{\psi, \fp}(X)
=
X^2+a_{\fp}X+\mu_{\fp}  p,
\end{equation}
where $\mu_{\fp} \in \F_q^\times$. Moreover, 
\begin{equation}\label{trace-bound}
\deg a_{\fp} \leq \frac{\deg p}{2}.
\end{equation}
Therefore
$|\pi_{\fp}|=|\pi'_{\fp}|=|p|^{\frac{1}{2}}$, where $\pi'_{\fp}$ denotes the conjugate of $\pi_{\fp}$ over $F$.
In particular,
\begin{equation}\label{eq-Thm2-1}
|p|=|\pi_{\fp}|\cdot |\pi'_{\fp}|=|1+d_1\alpha|\cdot |1+d_1\alpha'|=|d_1\alpha|\cdot |d_1\alpha'|=|d_1|^2\cdot |\alpha\alpha'|, 
\end{equation}
where $\alpha'$ denotes the conjugate of $\alpha$ over $F$.

We write $\alpha=a_1+a_2\sqrt{\Delta}$ for some $a_1, a_2\in A$. Note that $a_2\neq 0$ since $\pi\not\in A$. Then
\begin{equation}\label{eq-Thm2-2}
|\alpha\alpha'|=|a_1^2-a_2^2\Delta|. 
\end{equation}
The leading terms of $a_1^2$ and $a_2^2\Delta$, as polynomials in $T$, cannot cancel. Indeed, $a_1^2$ and $a_2^2$ have 
even degrees and their leading coefficients are squares in $\F_q^\times$, whereas $\Delta$ either has odd degree or its leading 
coefficient is not a square in $\F_q^\times$ (see Remark \ref{rem-inf}). This implies that 
\begin{equation}\label{eq-Thm2-3}
|a_1^2-a_2^2\Delta|\geq |\Delta|. 
\end{equation}

Combining Theorem \ref{thm1} with \eqref{eq-Thm2-1}, \eqref{eq-Thm2-2}, \eqref{eq-Thm2-3}, we obtain 
$$
|p|\geq |d_1|^2\cdot |\Delta|\gg_{\psi} |d_1|^2\cdot  \frac{\log |p|}{(\log\log |p|)^2},
$$
or, equivalently, 
$$
|d_1|\ll_{\psi} \frac{\sqrt{|p|}\cdot \log\log |p|}{\sqrt{\log|p|}}. 
$$
Now recall that $|d_1d_2|=|p|$. Hence 
$$
|d_2|\gg_{\psi} \frac{\sqrt{|p|\cdot \log|p|}}{\log\log |p|}. 
$$
This completes  the proof of Theorem \ref{thm2}. 
 
\section{Proof of Theorem \ref{thm3} }

Let $\psi$ be a Drinfeld module of rank $2$ over $F$ 
and 
 let $\fp\lhd A$ be a fixed prime where $\psi$ has good reduction. 
 As before, let $\psi\otimes \F_\fp$ be the 
reduction of $\psi$ at $\fp$.
 As we mentioned earlier, 
  the rings
 $A[\pi_\fp] \subseteq \cE_{\psi, \fp} \subseteq \cO_{F(\pi_\fp)}$ 
are $A$-orders in the imaginary quadratic extension $K_\fp := F(\pi_\fp)$ of $F$. 
Since $\fp$ varies in this section,  we now specify the dependence on $\fp$ of
$\cE_{\psi, \fp}$, $\Delta_{\fp}$, and $K_\fp$, as well as of all other relevant data.

Similarly to the $A$-module isomorphism
$\cO_{F(\pi_\fp)}/\cE_{\psi, \fp}\cong A/c_{\fp} A$ mentioned in the introduction,
there is an $A$-module isomorphism
$\cE_{\psi, \fp}/A[\pi_\fp]\cong A/b_{\fp}A$, where $b_\fp = b_{\psi, \fp} \in A$ is a monic polynomial. 
Comparing the discriminants of $\cE_\fp$ and $A[\pi_\fp]$ and remarking that the discriminant of  $A[\pi_\fp]$ is the discriminant of the polynomial $P_{\psi, \fp}$ of (\ref{char-poly}), 
we find a (not necessarily monic) generator 
$\delta_\fp = \delta_{\psi, \fp} \in A$ of the ideal $(\Delta_{\fp})$ which is uniquely determined by the relation
\begin{equation}\label{disc-relation}
a_{\fp}^2 - 4 \mu_{\fp} p = b_{\fp}^2 \delta_{\fp}.
\end{equation} 
Taking norms and using (\ref{trace-bound}), we deduce that
\begin{equation}\label{deg-bp-deltap}
|a_\fp^2 - 4 \mu_\fp p| = |b_\fp|^2 |\delta_\fp|.
\end{equation}
Thus to study the growth of the quotient $|\Delta_\fp|/|a_\fp^2 - 4 \mu_\fp p| = |\delta_\fp|/|a_\fp^2 - 4 \mu_\fp p|$ it suffices to study the growth
of $|b_\fp|$.
In particular, letting 
$f: (0, \infty) \longrightarrow (0, \infty)$ be a function such that  $\ds\lim_{x \rightarrow \infty} f(x) = \infty$,
we see that showing that
$$
\#\left\{
\fp  \in {\cal{P}}(\psi):
\deg p = x,
|\Delta_\fp| >
\frac{|a_\fp^2 - 4 \mu_\fp p|}{q^{f(\deg p)}} 
\right\}
\sim
\pi_F(x)
$$
is equivalent to showing that
\begin{equation}\label{bp-little-o}
\#\left\{
\fp  \in {\cal{P}}(\psi):
\deg p = x,
|b_\fp| \geq
q^{\frac{f(\deg p)}{2}} 
\right\}
=
\o(\pi_F(x))).
\end{equation}

To estimate from above the left hand side of (\ref{bp-little-o}), our strategy is to partition the primes $\fp$ according to
the value $m$ taken by $b_\fp$, and then to relax the equality $b_\fp = m$ to the divisibility $m \mid b_\fp$.
This strategy leads us to the problem of understanding 
the divisibility properties 
of the polynomials $b_{\mathfrak{p}}$, which 
 was already explored in \cite{CoPa15}, as  we recall below.
 
 Let $m \in A$ be a monic polynomial  
and denote by $F(\psi[m])$ the field obtained by adjoining to $F$ the elements of $\psi[m]$.
We obtain a finite, Galois extension of $F$, whose Galois group $\Gal(F(\psi[m])/F)$ may be viewed as a subgroup
of $\GL_2(A/m A)$. 
An important subfield of  $F(\psi[m])$
is
 $$
 J_m :=
 \left\{
 z \in F(\psi[m]):
 \sigma(z) = z \  \forall \sigma \in \Gal(F[\psi[m]/F) \ \text{a scalar element}
 \right\},
 $$
 which is itself a Galois extension of $F$
 whose Galois group $\Gal(J_m/F)$ may be viewed as a subgroup of $\PGL_2(A/m A)$.
 As  a consequence of \cite[Thm. 1]{CoPa15}, we have that, for any $\fp \in \cal{P}(\psi)$ with
 $\fp \nmid m$,
 \begin{equation}\label{bp-mod-m}
 m \mid b_\fp
 \
 \Leftrightarrow
 \
 \fp \ \text{splits completely in} \ J_m/F.
 \end{equation}
 Thus the divisibility $m \mid b_{\fp}$ may be studied via the Chebotarev density theorem.
 
 Indeed, denoting by 
 $$
 \Pi_1(x, J_m/F) :=
 \#\left\{
\fp \in {\cal{P}}(\psi):
\deg p = x,
\fp\ \text{splits completely in} \ J_m/F
 \right\},
$$
it was proved in \cite[Thm. 15]{CoPa15}
 that an effective version of the Chebotarev density theorem applied to the extension
 $J_m/F$ gives
 \begin{equation}\label{Jm-chebotarev}
\Pi_{1}(x, J_m/F)
=
\frac{c_{m}(x)}{[J_m : F]} \cdot \frac{q^x}{x}
+
\O_{\psi}\left( q^{\frac{x}{2}} \deg m\right),
\end{equation}
where
$$
c_m := \left[J_m \cap \overline{\F}_q : \F_q\right]
$$
and
\begin{equation}\label{cm-x-def}
c_m(x)
:=
\left\{
\begin{array}{cc}
c_{m} & \textrm{if } c_{m} \mid x, \\
  0  &    \textrm{otherwise}.
\end{array}
\right.
\end{equation}
Furthermore, as pointed out in \cite[Thm. 15]{CoPa15}, we have
that
\begin{equation}\label{cm-bound}
c_m \ll_{\psi} 1,
\end{equation}
regardless of the structure of $\End_{F^{\text{alg}}}(\psi)$.
Finally, 
under the assumption $\End_{F^{\text{alg}}}(\psi) = A$, which we now make,
 it was proved in \cite[Cor. 3]{CJ20} that
\begin{equation}\label{Jm-degree}
|m|^3
\ll_{\psi}
[J_m : F]
\leq
|m|^3.
\end{equation}
(Note that only the lower bound requires the assumption on the endomorphism ring.)

We also remark that the condition $m \mid b_\fp$ for some $\fp$ with $\deg p = x$ implies that 
\begin{equation}\label{m-divides-bp-bound}
\deg m \leq \frac{\deg \left(a_\fp^2 - 4 \mu_\fp p\right) }{2} \leq \frac{x}{2},
\end{equation}
since $\deg a_\fp \leq \frac{\deg p}{2}$.

Now let us apply the aforementioned strategy to the left hand side of (\ref{bp-little-o}).
After partitioning the primes $\fp$
according to the values taken by $b_\fp$,
relaxing the equality $b_{\mathfrak{p}} = m$ to the divisibility $m \mid b_\fp$,
and
using (\ref{bp-mod-m}) and (\ref{m-divides-bp-bound}),
we obtain that
\begin{eqnarray}\label{sum-large-bp}
\hspace*{-0.5cm}
\#\left\{
\fp  \in {\cal{P}}(\psi):
\deg p = x,
|b_\fp| \geq
q^{\frac{f(\deg p)}{2}} 
\right\}
&\leq&
\ds\sum_{
m \in A
\atop{
m \ \text{monic}
\atop{
\frac{f(x)}{2} \leq \deg m \leq \frac{x}{2}
}
}
}
\#\left\{
\fp  \in {\cal{P}}(\psi):
\deg p = x,
m \mid b_\fp
\right\}
\nonumber
\\
&=&
\ds\sum_{
m \in A
\atop{
m \ \text{monic}
\atop{
\frac{f(x)}{2} \leq \deg m \leq y
}
}
}
\Pi_1(x, J_m/F)
\nonumber
\\
&+&
\ds\sum_{
m \in A
\atop{
m \ \text{monic}
\atop{
y <  \deg m \leq \frac{x}{2}
}
}
}
\#\left\{
\fp \in {\cal{P}}(\psi):
\deg p = x,
m^2 \mid \left(a_\fp^2 - 4 \mu_{\mathfrak{p}} p\right)
\right\},
\end{eqnarray}
where $y = y(x)$ is a parameter satisfying $f(x) \leq 2 y \leq x$,  
with the role of delimitating the range for which the interpretation of the divisibility $m \mid b_\fp$
as a splitting completely condition as in  (\ref{bp-mod-m})
can be quantified via a direct application of the Chebotarev density theorem.
Using 
(\ref{Jm-chebotarev}), (\ref{cm-bound}), (\ref{Jm-degree}), 
and
proceeding similarly to the proof of \cite[(39)]{CoPa15}, 
we deduce that
\begin{eqnarray}
\label{sum-chebotarev-estimated1}
\ds\sum_{
m \in A
\atop{
m \ \text{monic}
\atop{
\frac{f(x)}{2} \leq \deg m \leq y
}
}
}
\Pi_1(x, J_m/F)
&=&
\frac{q^x}{x}
\ds\sum_{
m \in A^{(1)}
\atop{
\frac{f(x)}{2} \leq \deg m
}
}
\frac{c_m(x)}{[J_m : F]} 
+
\O_{\psi}\left(q^{\frac{x}{2} +y}\right)
+
\O_{\psi}\left(q^{x-2 y}\right)
\\
\label{sum-chebotarev-estimated2}
&\ll_{\psi}&
q^{x - f(x)} + q^{\frac{x}{2} + y}.
\end{eqnarray}
In particular, we see that we need that the parameter $y = y(x)$ satisfy $y < \frac{x}{2}$.

To estimate the sum over $m \in A$ with $y <  \deg m \leq \frac{x}{2}$,  
we appeal to the more elaborate application of 
(\ref{Jm-chebotarev}), (\ref{cm-bound}), (\ref{Jm-degree}) via the Square Sieve.
Specifically, proceeding similarly to the proof of \cite[(41)]{CoPa15}, we deduce that,
for any $\varepsilon > 0$,
\begin{eqnarray}\label{sum-square-estimated}
\ds\sum_{
m \in A
\atop{
m \ \text{monic}
\atop{
y <  \deg m \leq \frac{x}{2}
}
}
}
\#\left\{
\fp \in {\cal{P}}(\psi):
\deg p = x,
m^2 \mid \left(a_\fp^2 - 4 \mu_{\mathfrak{p}} p\right)
\right\}
\ll_{\psi, \varepsilon}
q^{\frac{15 x}{8} - 2 y + x \varepsilon} x^3.
\end{eqnarray}
(Note that while in \cite{CoPa15} the occuring $m$ is squarefree, this arithmetic feature of $m$ is not used in the estimates therein; as such, (\ref{sum-chebotarev-estimated2}) and (\ref{sum-square-estimated}) hold also when $m$ has square factors.)

Choosing $y = y(x) := \frac{(11 + \varepsilon) x}{24}$
and recalling that $\ds\lim_{x \rightarrow \infty} f(x) = \infty$, we deduce from 
(\ref{sum-large-bp}), (\ref{sum-chebotarev-estimated2}),
and (\ref{sum-square-estimated})
that
\begin{eqnarray}\label{main-error}
\#\left\{
\fp  \in {\cal{P}}(\psi):
\deg p = x,
|b_\fp| \geq
q^{\frac{f(\deg p)}{2}} 
\right\}
\ll_{\psi, \varepsilon}
q^{x - f(x)} + q^{\frac{23 x}{24} + \varepsilon x}
=
 \o \left(\frac{q^x}{x}\right),
\end{eqnarray}
which confirms (\ref{bp-little-o}).

Now let us prove that the set
$
\left\{
\fp \in  \cP(\psi):
 |\Delta_{\fp}|
 >
\frac{|a_\fp^2 - 4 \mu_\fp p|}{q^{f(\deg p)}} 
\right\}
$
has Dirichlet density $1$.
For this, let $s > 1$ and consider the sum
\begin{eqnarray*}
\ds\sum_{
\fp \in {\cal{P}}(\psi)
\atop{
|\Delta_\fp| \leq \frac{|a_\fp^2 - 4 \mu_\fp p|}{q^{f(\deg p)}}
} 
}
q^{-s \deg p}
&=&
\ds\sum_{x \geq 1}
q^{-s x}
\#\left\{
\fp \in {\cal{P}}(\psi):
\deg p = x,
|\Delta_\fp| \leq \frac{|a_\fp^2 - 4 \mu_\fp p|}{q^{f(\deg p)}}
\right\}
\\
&=&
\ds\sum_{x \geq 1}
q^{-s x}
\#\left\{
\fp \in {\cal{P}}(\psi):
\deg p = x,
|b_\fp| \geq q^{\frac{f(x)}{2}}
\right\}.
\end{eqnarray*}
By (\ref{sum-large-bp}), (\ref{sum-chebotarev-estimated1}), (\ref{sum-square-estimated}), and the earlier choice of
$y$ we obtain that the above is
\begin{eqnarray*}
&\leq&
\ds\sum_{x \geq 1}
\frac{q^{(1-s) x}}{x}
\ds\sum_{
m \in A
\atop{
m \ \text{monic}
\atop{
\frac{f(x)}{2} \leq \deg m \leq \frac{x}{2}
}
}
}
\frac{c_m(x)}{[J_m:F]}
+
\O_{\psi, \varepsilon}\left(
\ds\sum_{x \geq 1}
q^{\left(\frac{23}{24} + \varepsilon - s\right) x }
\right).
\end{eqnarray*}

To estimate the first term, 
we first use (\ref{cm-x-def}) to rewrite $c_m(x)$:
\begin{eqnarray*}
T_1
&:=&
\ds\sum_{x \geq 1}
\frac{q^{(1-s) x}}{x}
\ds\sum_{
m \in A
\atop{
m \ \text{monic}
\atop{
\frac{f(x)}{2} \leq \deg m \leq \frac{x}{2}
}
}
}
\frac{c_m(x)}{[J_m:F]}
\\
&=&
\ds\sum_{x \geq 1}
\frac{q^{(1-s) x}}{x}
\ds\sum_{
m \in A
\atop{
m \ \text{monic}
\atop{
\frac{f(x)}{2} \leq \deg m \leq \frac{x}{2}
\atop{
c_m \mid x
}
}
}
}
\frac{c_m}{[J_m:F]}
\\
&=&
\ds\sum_{
m \in A
\atop{
m \ \text{monic}
}
}
\ds\sum_{
j \geq 1
\atop{
\frac{f(c_m j)}{2} \leq \deg m \leq \frac{c_m j}{2}
}
}
\frac{q^{(1-s) c_m j}}{j [J_m : F]}.
\end{eqnarray*}
Next we fix $M > 0$ and observe that, since $\ds\lim_{x \rightarrow \infty} f(x) = \infty$,
there exists $n(M) \in \N$ such that
\begin{equation}\label{f(n)}
f(n) > M \quad \forall n \geq n(M).
\end{equation}
Now we split the sum over $j$ in our last reformulation of $T_1$ according to whether
$c_m j \geq n(M)$ or $c_m j < n(M)$. 
For the first range,
we obtain
\begin{eqnarray*}
T_{1, 1}
&:=&
\ds\sum_{m \in A
\atop{m \ \text{monic}}
}
\ds\sum_{
j \geq 1
\atop{
c_m j \geq n(M)
\atop{
\frac{f(c_m j)}{2} \leq \deg m \leq \frac{c_m j}{2}
}
}
}
\frac{q^{(1-s) c_m j}}{j [J_m : F]}
\\
&\leq&
\ds\sum_{
m \in A
\atop{
m \ \text{monic}
\atop{
\frac{M}{2} \leq \deg m 
}
}
}
\frac{1}{[J_m : F]}
\ds\sum_{
j \geq 1
}
\frac{q^{(1-s) c_m j}}{j}
\\
&\ll_{\psi}&
\ds\sum_{
m \in A
\atop{
m \ \text{monic}
\atop{
\frac{M}{2} \leq \deg m 
}
}
}
\frac{1}{[J_m : F]}
\ds\sum_{
j \geq 1
}
\frac{q^{(1-s)  j}}{j}
\\
&\ll_{\psi}&
\ds\sum_{
m \in A
\atop{
m \ \text{monic}
\atop{
\frac{M}{2} \leq \deg m
}
}
}
\frac{
\log \deg m + \log \log q
}{
|m|^3
}
\ds\sum_{j \geq 1}
\frac{q^{(1-s) 
 j }}{j}
 \\
 &\ll&
 \frac{
 \log \frac{M}{2}
 }{
 q^{M} \log q
 }
 \
 \left|
 \log \left(
 1 - q^{1-s}
 \right)
 \right|,
\end{eqnarray*}
where 
we used
(\ref{f(n)}) to pass to the second line,
(\ref{cm-bound}) to pass to the third line,
(\ref{Jm-degree}) to pass to the fourth line, 
and $s > 1$, together with  \cite[Lem. 2.2]{CoSh15} 
 to pass to the fifth line.
It follows that
\begin{equation}\label{T11-estimate}
\ds\lim_{s \rightarrow 1+}
\frac{T_{1, 1}}{- \log \left(1 - q^{1-s}\right)}
\ll_{\psi}
\frac{\log \frac{M}{2} }{q^{M} \log q}.
\end{equation}
For the second range, we investigate
\begin{eqnarray*}
T_{1, 2}
&:=&
\ds\sum_{m \in A
\atop{
m \ \text{monic}
}
}
\ds\sum_{
j \geq 1
\atop{
c_m j < n(M)
\atop{
\frac{f(c_m j)}{2} \leq \deg m \leq \frac{c_m j}{2}
}
}
}
\frac{q^{(1-s) c_m j}}{j [J_m : F]}
\end{eqnarray*}
and observe that it is a finite sum 
since $\deg m \leq \frac{c_m j}{2} \ll_{\psi} j$ (recalling (\ref{cm-bound}))
and
since $j < n(M)$.
Observing that
$\ds\lim_{s \rightarrow 1+} \frac{q^{(1-s) \alpha}}{\log \left(1 - q^{1-s}\right)}= 0$ 
for any $\alpha \geq 1$,
it follows that
\begin{equation}\label{T12-estimate}
\ds\lim_{s \rightarrow 1+}
\frac{T_{1, 2}}{- \log \left(1 - q^{1-s}\right)}
= 0.
\end{equation}

To estimate the remaining term
\begin{eqnarray*}
T_2
&:=&
\ds\sum_{x \geq 1}
q^{\left(\frac{23}{24}  + \varepsilon - s\right) x },
\end{eqnarray*}
observe that
\begin{eqnarray}\label{T2-estimate}
\ds\lim_{s \rightarrow 1+}
\frac{T_2}{- \log \left(1 - q^{1-s}\right)}
=
- 
\ds\lim_{s \rightarrow 1+}
\frac{q^{\left(\frac{23}{24} + \varepsilon - s\right) x}}{
\left(
1 - q^{\left(\frac{23}{24} + \varepsilon - s\right)}
\right)
\log
\left(
1 - q^{(1 -s) x}
\right)
}
= 0.
\end{eqnarray}

We now take $M \rightarrow \infty$ in (\ref{f(n)}) and put together 
(\ref{T11-estimate}), (\ref{T12-estimate}),  (\ref{T2-estimate}),
completing the proof of Theorem \ref{thm3}.

\section*{Acknowledgments}
We thank Zeev Rudnick for his comments on an earlier version of Theorem \ref{thm1}
which prompted us to obtain an improved bound.



\begin{thebibliography}{GP19b}
	
	\bibitem[Bey09]{Be09}
	Jeffrey Beyerl.
	\newblock {\em Binary quadratic forms over $\mathbb{F}_q[T]$ and principal
		ideal domains}.
	\newblock 2009.
	\newblock Thesis (Masters)--Clemson University, U.S.A..
	
	\bibitem[Bre10]{Br}
	Florian Breuer.
	\newblock Torsion bounds for elliptic curves and {D}rinfeld modules.
	\newblock {\em J. Number Theory}, 130(5):1241--1250, 2010.
	
	\bibitem[CF19]{CoFi19}
	Alina~Carmen Cojocaru and Matthew Fitzpatrick.
	\newblock The absolute discriminant of the endomorphism ring of
most reductions of a non-CM elliptic curve is close to maximal. 
	\newblock {\em Preprint} 2019.
	
	\bibitem[Cox89]{Co89}
	David~A. Cox.
	\newblock {\em Primes of the form {$x^2 + ny^2$}}.
	\newblock A Wiley-Interscience Publication. John Wiley \& Sons, Inc., New York,
	1989.
	\newblock Fermat, class field theory and complex multiplication.
	
	\bibitem[CJ20]{CJ20}
	Alina~Carmen Cojocaru and Nathan Jones.
	\newblock Degree bounds for projective division fields associated to elliptic modules with a trivial endomorphism ring.
	\newblock {\em Preprint} 2020.
	
	
	\bibitem[CP15]{CoPa15}
	Alina~Carmen Cojocaru and Mihran Papikian.
	\newblock Drinfeld modules, {F}robenius endomorphisms, and {CM}-liftings.
	\newblock {\em Int. Math. Res. Not. IMRN}, (17):7787--7825, 2015.
	
	\bibitem[CS15]{CoSh15}
	Alina~Carmen Cojocaru and Andrew~Michael Shulman.
	\newblock The distribution of the first elementary divisor of the reductions of
	a generic {D}rinfeld module of arbitrary rank.
	\newblock {\em Canad. J. Math.}, 67(6):1326--1357, 2015.
	
	\bibitem[Dri74]{Dr74}
	Vladimir Gershonovich Drinfeld.
	\newblock Elliptic modules.
	\newblock {\em Mat. Sb. (N.S.)}, 94(136):594--627, 656, 1974.
	
	\bibitem[Gek83]{Ge83}
	Ernst-Ulrich Gekeler.
	\newblock Zur {A}rithmetik von {D}rinfeld-{M}oduln.
	\newblock {\em Math. Ann.}, 262(2):167--182, 1983.
	
	\bibitem[Gek91]{Ge91}
	Ernst-Ulrich Gekeler.
	\newblock On finite {D}rinfeld modules.
	\newblock {\em J. Algebra}, 141(1):187--203, 1991.
	
	\bibitem[Gek99]{Ge97}
	Ernst-Ulrich Gekeler.
	\newblock Some new results on modular forms for
	{${\mathrm{GL}}(2,\mathbb{F}_q[T])$}.
	\newblock In {\em Recent progress in algebra ({T}aejon/{S}eoul, 1997)}, volume
	224 of {\em Contemp. Math.}, pages 111--141. Amer. Math. Soc., Providence,
	RI, 1999.
	
	\bibitem[Gos96]{Go96}
	David Goss.
	\newblock {\em Basic structures of function field arithmetic}, volume~35 of
	{\em Ergebnisse der Mathematik und ihrer Grenzgebiete (3) [Results in
		Mathematics and Related Areas (3)]}.
	\newblock Springer-Verlag, Berlin, 1996.
	
	\bibitem[GP19a]{GaPa19}
	Sumita Garai and Mihran Papikian.
	\newblock Computing endomorphism rings and {Frobenius} matrices of {Drinfeld}
	modules.
	\newblock {\em J. Number Theory}, 2019.
	\newblock DOI: 10.1016/j.jnt.2019.11.018.
	
	\bibitem[GP19b]{GaPa18}
	Sumita Garai and Mihran Papikian.
	\newblock Endomorphism rings of reductions of {Drinfeld} modules.
	\newblock {\em J. Number Theory}, 2019.
	\newblock DOI: 10.1016/j.jnt.2019.02.008.
	
	\bibitem[Ros02]{Ro02}
	Michael Rosen.
	\newblock {\em Number theory in function fields}, volume 210 of {\em Graduate
		Texts in Mathematics}.
	\newblock Springer-Verlag, New York, 2002.
	
	\bibitem[Sch91]{Sc91}
	Ren\'{e} Schoof.
	\newblock The exponents of the groups of points on the reductions of an
	elliptic curve.
	\newblock In {\em Arithmetic algebraic geometry ({T}exel, 1989)}, volume~89 of
	{\em Progr. Math.}, pages 325--335. Birkh\"{a}user Boston, Boston, MA, 1991.
	
	\bibitem[Yu95a]{Yu95bis}
	Jiu-Kang Yu.
	\newblock A class number relation over function fields.
	\newblock {\em J. Number Theory}, 54(2):318--340, 1995.
	
	\bibitem[Yu95b]{Yu95}
	Jiu-Kang Yu.
	\newblock Isogenies of {D}rinfeld modules over finite fields.
	\newblock {\em J. Number Theory}, 54(1):161--171, 1995.
	
\end{thebibliography}
\end{document}